\documentclass [a4paper,12pt]{article}
\usepackage{amsmath}
\usepackage{amsfonts}
\usepackage{amssymb}
\usepackage{latexsym}
\usepackage{fancyhdr,lastpage}
\usepackage{subfigure}
\usepackage{graphicx}
\usepackage[font=small,labelfont=bf]{caption}
\usepackage[font=small,labelfont=]{subfig}
\usepackage{color}
\begin{document}

\newtheorem{theorem}{Theorem}[section]
\newtheorem{lemma}[theorem]{Lemma}
\newtheorem{proposition}[theorem]{Proposition}
\newtheorem{corollary}[theorem]{Corollary}

\newenvironment{proof}[1][Proof]{\begin{trivlist}
\item[\hskip \labelsep {\bfseries #1}]}{\end{trivlist}}
\newenvironment{definition}[1][Definition]{\begin{trivlist}
\item[\hskip \labelsep {\bfseries #1}]}{\end{trivlist}}
\newenvironment{example}[1][Example]{\begin{trivlist}
\item[\hskip \labelsep {\bfseries #1}]}{\end{trivlist}}
\newenvironment{remark}[1][Remark]{\begin{trivlist}
\item[\hskip \labelsep {\bfseries #1}]}{\end{trivlist}}
\newenvironment{question}[1][Question]{\begin{trivlist}
\item[\hskip \labelsep {\bfseries #1}]}{\end{trivlist}}

\newcommand{\qed}{\nobreak \ifvmode \relax \else
\ifdim\lastskip<1.5em \hskip-\lastskip
\hskip1.5em plus0em minus0.5em \fi \nobreak
\vrule height0.75em width0.5em depth0.25em\fi}

 % Please minimize the usage of "newtheorem", "newcommand", and use
 % equation numbers only situation when they provide essential convenience
 % Try to avoid defining your own macros

\title{On the persistence properties of the cross-coupled Camassa-Holm system}

% Place all authors' names in [ ] shown as running head;
% No more than 40 letters. Leave { } empty
% Please use `and' to connect the last two names if appliable

% Enter the first author's name and address:
\author{David Henry \\ {\small School of Mathematical Sciences, University College Cork, Cork, Ireland}
\\
Darryl D. Holm \\ {\small Mathematics Department, Imperial College
London, SW7 2AZ, UK}
\\
Rossen Ivanov \\ {\small School of Mathematical Sciences, Dublin
Institute of Technology},\\ {\small Kevin Street, Dublin 8,
Ireland}
} % Do not forget to end the {\footnotesize by the sign }

\date{}
% The name of the associate editor will be entered by an editorial staff

\maketitle

%The abstract of your paper

\begin{abstract}
\noindent In this paper we examine the evolution of solutions, that initially have compact
support,  of a recently-derived system \cite{CHIP} of cross-coupled Camassa-Holm equations.
The analytical methods which we employ provide a full picture for the persistence of compact support for the momenta. For  solutions of the system itself, the answer is more convoluted, and we determine when the compactness of the support is lost, replaced instead by an exponential decay rate.
\end{abstract}

\section{Introduction}\indent
This paper is concerned with the persistence of compact support in solutions to
a recently derived cross-coupled Camassa-Holm (CCCH) equation
\cite{CHIP}, which is given by \begin{subequations}\label{sysa}
\begin{align}
m_t+2v_xm+vm_x&=0 \label{CH1} \\
n_t+2u_xn+un_x&=0, \label{CH2},
\end{align}
\end{subequations}
where $m=u-u_{xx}$ and $n=v-v_{xx}$. This system generalises the
celebrated Camassa-Holm (CH) equation \cite{Cam}, since for $u=v$
the system \eqref{sys} reduces to two copies of the CH equation

$$ m_t+2u_xm+um_x=0. $$
The CH equation models a variety of phenomena, including the
propagation of unidirectional shallow water waves over a flat bed
\cite{Cam,DGH,HoIv,John,Iv1}. The CH equation possesses a very
rich structure, being an integrable infinite-dimensional
Hamiltonian system with a bi-Hamiltonian structure and an infinity
of conservation laws \cite{Cam,ConGerIv,Iv2}. It also has a
geometric interpretation as a re-expression of the geodesic flow
on the diffeomorphism group of the circle \cite{HoMaRa1998}. One
of the most interesting features of the CH equation, perhaps, is
the rich variety of solutions it admits. Some solutions exist
globally, whereas others exist only for a finite length of time,
modelling wave breaking \cite{ConMcK,ConEschActa}.

The CCCH equation can be derived from a variational principle as a
n Euler-Lagrange system of equations for the Lagrangian
$$
l(u,v)  = \int_{\mathbb{R}} \left(uv + u_xv_x \right) \text{ d}x.
$$ Alternatively it can be formulated as a two-component system of Euler-Poincar\'e (EP)
equations in one dimension on $\mathbb{R}$ as follows, $$
{\partial_t}m= -\, {\rm ad}^*_{\delta h/\delta m} m =
-\,(vm)_{x}-mv_x \quad\hbox{with}\quad v := \frac{\delta h}{\delta
m} = K*n ,
$$
$$
{\partial_t}n= -\, {\rm ad}^*_{\delta h/\delta n} n =
-\,(un)_{x}-nu_x \quad\hbox{with}\quad u := \frac{\delta h}{\delta
n} = K*m ,
$$ with $K(x,y)=\frac12 e^{-|x-y|}$  being the Green function of the Helmholtz operator, and
$h$ being the Hamiltonian \begin{eqnarray*} h(n,m)  =
\int_{\mathbb{R}} n\,K*m  \text{ d}x = \!\!\int_{\mathbb{R}}
m\,K*n  \text{ d}x.
\end{eqnarray*}

This Hamiltonian system has {\bf two}-component singular momentum
map \cite{HoMa2004}
$$m(x,t)=\sum_{a=1}^{M}m_{a}(t)\,\delta(x-q_{a}(t)), \qquad
 n(x,t)=\sum_{b=1}^{N}n_{b}(t)\,\delta(x-r_{b}(t)).$$

The $M=N=1$ case is very simple for analysis \cite{CHIP}. If  the
initial conditions are $m_1(0)>0$ and $n_1(0)>0$ then one observes
the so-called {\it waltzing} motion.  It could be noted that for
half of the waltzing period (half cycle) the two types of peakons
exchange momentum amplitudes - see Fig. 1. The explicit solutions
as well as other examples with waltzing peakons and compactons are
given in \cite{CHIP}.

\begin{figure}[t]
\begin{center}
\includegraphics*[width=0.8\textwidth]{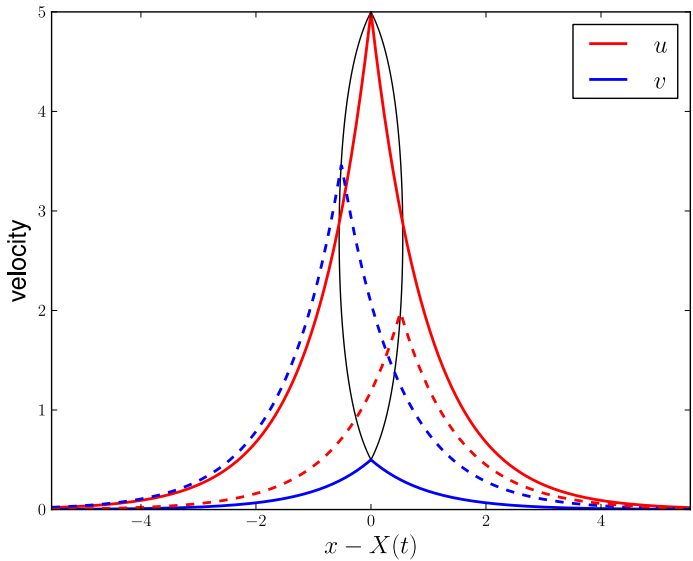}
\end{center}
\caption{\label{peakon_orbits} Plot showing velocity fields of a
peakon-peakon pair with $m_1(0)=10$, $n_1(0)=1$ (solid lines). The
dotted path indicates the subsequent path of the two peaks in the
frame travelling at the particles mean velocity. For these initial
conditions the total period for one orbit of the cycle is $T=3.6$.
Also shown is the form of the two peakons at subsequent times
$t=0.45+1.8n$, $n\in \mathbb{Z}$. }
\end{figure}

The aim of this study is to analyse the persistence of compact
support for solutions of the system \eqref{sys}. In particular, we
will examine whether the solution $m,n$, and in turn $u,v$, of
\eqref{CH1}-\eqref{CH2}, which initially have compact support,
will continue to do so as they evolve. Solutions of the system
which have compact support can be viewed as localized
disturbances, and whether a ``disturbance'' which is initially
localized propagates with a finite, or infinite speed, is a matter
of great interest. We will see that some solutions will remain
compactly supported at all future times of their existence, while
others solution display an infinite speed of propagation and
instantly lose their compact support. These results have analogues
in the CH case, which is simply the CH equation
\cite{Con1,HenJNMP,HenDCD}.

\section{Preliminaries}
We may re-express equation \eqref{sysa}  in terms of
$u$ and $v$ as follows
\begin{subequations}\label{sys}
\begin{align}
u_t-u_{xxt}+2v_xu-2v_xu_{xx}+vu_x-vu_{xxx}&=0, \label{CH1a} \\
v_t-v_{xxt}+2u_xv-2u_xv_{xx}+uv_x-uv_{xxx}&=0. \label{CH2a}
\end{align}
\end{subequations}

From this form of the equations one observes that there are no
terms with self-interaction (e.g. $u u_x$, $u_x u_{xx}$,
$uu_{xxx}$ etc.) which justifies the name 'cross-coupled'.

If $p(x)=\frac{1}{2}e^{-|x|},\ x\in\mathbb R$, then
$(1-\partial_x^2)^{-1}f=p\ast f$ for all $f\in L^2(\mathbb R)$ and
so $p \ast m=u$, $p \ast n=v$, where $\ast$ denotes convolution in
the spatial variable. Indeed,
 \begin{equation}
u(x)=\frac 12e^{-x}\int_{-\infty}^x e^ym(y)\,dy + \frac 12
e^x\int_x^{\infty}e^{-y}m(y)\,dy. \label{udecomp} \end{equation}

\begin{equation}
u_x(x)=-\frac 12e^{-x}\int_{-\infty}^x e^ym(y)\,dy + \frac 12
e^x\int_x^{\infty}e^{-y}m(y)\,dy. \label{uxdecomp} \end{equation}

In other words, if we denote by $I_1(x)$ and $I_2(x)$ the
integrals appearing in the first and the second term of
(\ref{udecomp}), we have

\begin{equation}
u=I_1+I_2, \qquad u_x=-I_1+I_2. \label{uANDuxdecomp}
\end{equation}

Applying the convolution operator to equation \eqref{sys} we can
re-express it in the form of a conservation law \begin{equation}
\label{CH1'} (u+v)_t+\partial_x\left(uv+ p \ast \left(2uv+u_xv_x
\right)\right)=0,\quad x\in \mathbb R,\ t\geq0,
\end{equation}  Thus $L=u+v$ is a density of the conserved momentum $\int(m+n)dx $.
The representation (\ref{CH1'}) agrees with the CH reduction when
$u=v$, cf. \cite{HenJNMP}.

The Hamiltonian $$H=\int(uv+u_xv_x)dx $$ (in terms of $u$ and $v$)
is of course another conserved quantity, the 'energy' of the
system, see more details in \cite{CHIP}.

%Of course, any solution of \eqref{sys} will also solve of
%\eqref{sys'}, but in general the converse will not hold: the
%solutions of \eqref{sys'} solve \eqref{sys} in a weak sense.

One can directly observe that (\ref{sys}) can be complexified in a
natural way if the variables $u$, $v$ are assumed complex, while
the independent variables $x$, $t$ are still real. Such a
complexified system is remarkable with the fact that it admits the
obvious reduction $u=\bar{v}$ which leads to a single scalar
complex equation:
\begin{equation}\label{complex}
u_t-u_{xxt}+2\bar{u}_xu-2\bar{u}_xu_{xx}+\bar{u}u_x-\bar{u}
u_{xxx}=0.
\end{equation}

This is a geodesic equation for a complex $H^1$ metric, given by
the Hamiltonian $H=\frac12\int(|u|^2+|u_x|^2)dx$.

Of course, if one reverts to real dependent variables according to
$u=r+is$ then (\ref{complex}) leads to the coupled system
\begin{subequations}\label{sys rs}
\begin{align}
r_t-r_{xxt}+2(rr_x+ss_x)-2(r_xr_{xx}+s_xs_{xx})-(rr_{xxx}+ss_{xxx})&=0, \label{CHr} \\
s_t-s_{xxt}+r_xs-rs_x
-2(r_xs_{xx}-s_xr_{xx})-(rs_{xxx}-sr_{xxx})&=0. \label{CHs}
\end{align}
\end{subequations}

Unless it is explicitly specified that the variables $(u,v)$ are
complex, we assume that they are real.

\section{Results}
In the following we let $T=T(u_0,v_0)>0$ denote the
maximal existence time of the solutions $u(x,t),v(x,t)$ to the
system \eqref{sys} with the given initial data $u_0(x)$ and
$v_0(x)$.

\subsection{Persistence of compact support for the momenta}
For the following, the flow prescribed by the system
\eqref{sysa}
is given by the two families of diffeomorphisms {$\left\{
\varphi(\cdot,t) \right\}_{t\in[0,T)}$}, {$\left\{ \xi(\cdot,t)
\right\}_{t\in[0,T)}$}
 as follows:
\begin{align}\label{diff}
\left\{\begin{array}{l}
 \varphi_t(x,t)=v(\varphi(x,t),t), \\
\varphi(x,0)=x,
\end{array}\right.
\qquad
\begin{array}{l}
 \xi_t(x,t)=u(\xi(x,t),t),\\ \xi(x,0)=x.
\end{array}
\end{align}
Solving \eqref{diff}, we get
\begin{equation}\label{incr}
\varphi_x(x,t)=e^{\int_0^t v_x(\varphi(x,s),s)ds},\ \xi_x(x,t)=e^{\int_0^t u_x(\xi(x,s),s)ds}>0,
\end{equation}
hence $\varphi(\cdot,t),\xi(\cdot,t)$ are increasing functions.
\begin{lemma}
\label{positivity} Assume that $u_0$ and $v_0$ are such that
$m_0=u_0-u_{0,xx}$ and $n_0=v_0-v_{0,xx}$ are nonnegative
(nonpositive) for $x\in \mathbb R$. Then $m(x,t)$ and $n(x,t)$
remain nonnegative (nonpositive) for all $t\in[0,T)$.
\end{lemma}

\begin{proof}
 It follows from   \eqref{sysa} that
\begin{eqnarray*}
\frac{d\phantom{t}}{dt}m(\varphi(x,t),t)\varphi_x^2(x,t)=m_t\varphi_x^2+m_x\varphi_t\varphi_x^2+2m\varphi_x\varphi_{xt}
\\ =(m_t+2v_xm+vm_x)\varphi_x^2=0,
\end{eqnarray*}
and
\begin{eqnarray*}
\frac{d\phantom{t}}{dt}n(\xi(x,t),t)\xi_x^2(x,t)=n_t\xi_x^2+n_x\xi_t\xi_x^2+2m\xi_x\xi_{xt}
\\ =(n_t+2u_xn+un_x)\xi_x^2=0.
\end{eqnarray*}
Therefore
\begin{equation}\label{int}
m(\varphi(x,t),t)\varphi_x^2(x,t)=m_0(x), \quad
n(\xi(x,t),t)\xi_x^2(x,t)=n_0(x).
\end{equation}
Now, since  $m_0(x),n_0(x)$ are nonnegative (nonpositive) then
$m(x,t)$ and $n(x,t)$ remain nonnegative (nonpositive) for all
$t\in[0,T)$. \qed
\end{proof}

\begin{lemma}
\label{mcompact} Assume that $u_0$ is such that $m_0=u_0-u_{0,xx}$
has compact support, contained in the interval
$[\alpha_{m_0},\beta_{m_0}]$ say, then for any $t\in[0,T)$, the
function $x \mapsto m(x,t) $ has compact support contained in the
interval $[\varphi(\alpha_{m_0},t),\varphi(\beta_{m_0},t)]$ for
all $t\in[0,T)$. Similarly, if $n_0=v_0-v_{0,xx}$ has compact
support,  then the function $x \mapsto n(x,t) $ is compactly
supported for all $t\in[0,T)$.
\end{lemma}
\begin{proof}
From (\ref{int}) and from the assumption that  $m_0(x)$ is
supported in the compact interval  $[\alpha_{m_0},\beta_{m_0}]$,
it follows directly
that $m(\cdot,t)$ are compactly supported, with support
contained in the interval
$[\varphi(\alpha_{m_0},t),\varphi(\beta_{m_0},t)]$, for all
$t\in[0,T)$. Similar reasoning applies to $n_0$. \qed
\end{proof}
Relation (\ref{int}) represents the conservation of momentum in
the physical variables cf. discussion in \cite{CHIP}.

\subsection{On the evolution of $(u,v)$}
In this subsection we are going to examine the general behaviour of
the solution $(u,v)$ of \eqref{sys} which is initially compactly supported. The
following Theorem provides us with some information about the
asymptotic behavior of the solution as it evolves over time - in
general, the solution has an exponential decay as $|x|\rightarrow
\infty$ for all future times $t\in[0,T)$.
\begin{theorem} \label{asymptotics}
 Let $(u,v)$ be a nontrivial solution of \eqref{sys},
with maximal time of existence $T>0$, and which is initially
compactly supported on an interval $\mathcal
I_0=[\alpha_{u_0},\beta_{u_0}]\times[\alpha_{v_0},\beta_{v_0}]$.
Then we have
\begin{align}
u(x,t)=
\left\{
\begin{array}{ll}
\frac 12 E^u_+(t)e^{-x} & \mbox{ for } x> \xi(\beta_{u_0},t), \\
\frac 12 E^u_-(t)e^{x} & \mbox{ for } x<\xi(\alpha_{u_0},t),
\end{array}
\right.,
\\
v(x,t)=
\left\{
\begin{array}{ll}
\frac 12 E^v_+(t)e^{-x} & \mbox{ for } x> \varphi(\beta_{v_0},t), \\
\frac 12 E^v_-(t)e^{x} & \mbox{ for } x<\varphi(\alpha{v_0},t),
\end{array}
\right.
\end{align}
where $\alpha,\beta$ are defined in \eqref{alphabet} below,
and  $E^u_-,E^u_+,E^v_-,E^v_+$ are continuous functions, with
$E^u_+(0)=E^v_+(0)=E^u_-(0)=E^v_-(0)=0$.

\end{theorem}
\begin{proof}
Firstly, if $(u_0,v_0)$ is initially supported on the compact interval $\mathcal I_0=[\alpha_{u_0},\beta_{u_0}]\times[\alpha_{v_0},\beta_{v_0}]$ then so too is $m_0$, and from the proof Lemma~\ref{mcompact} it follows that $\left(m(\cdot,t),n(\cdot,t)\right)$ is compactly supported, with its support contained in the interval $\mathcal I_t=[\xi(\alpha,t),\xi(\beta,t)]\times[\varphi(\alpha,t),\varphi(\beta,t)]$ for fixed $t\in[0,T)$.  Here
\begin{equation}\label{alphabet}
\alpha=\min\{\alpha_{u_0},\alpha_{v_0}\},\ \beta=\max\{\beta_{u_0},\beta_{v_0}\}.
\end{equation}
We use the  relation $u=p\ast m$ to write
\[
u(x)=\frac 12 e^{-x}\int_{-\infty}^x e^ym(y)\,dy + \frac 12 e^x\int_x^{\infty}e^{-y}m(y)\,dy,
\]
and then we define our functions
\begin{equation}\label{EuDef}
E^u_+(t)=\int_{\xi(\alpha,t)}^{\xi(\beta,t)}e^ym(y,t) \,dy \quad
\mbox{ and }
E^u_-(t)=\int_{\xi(\alpha,t)}^{\xi(\beta,t)}e^{-y}m(y,t)dy.
\end{equation}
We have
\begin{equation}\label{uEpm}\begin{array}{ll}
u(x,t)= \frac{1}{2}e^{-x}E^u_+(t), \quad & x> \xi(\beta,t), \\
u(x,t)= \frac{1}{2}e^{x}E^u_-(t), \quad & x<\xi(\alpha,t),
\end{array}
\end{equation}
and therefore from differentiating \eqref{uEpm} directly we get
\begin{equation*}\begin{array}{ll}
\frac{1}{2}e^{-x}E^u_+(t)=u(x,t)=-u_x(x,t)=u_{xx}(x,t), \quad & x> \xi(\beta,t), \\
\frac{1}{2}e^{x}E^u_-(t)=u(x,t)=u_x(x,t)=u_{xx}(x,t),\quad & x<\xi(\alpha,t).
\end{array}
\end{equation*}
Since $u(\cdot,0)$ is supported in the interval $[\alpha,\beta]$,
we have  $E^u_+(0)=E^u_-(0)=0$, as we can see by taking
integration by parts where the boundary terms vanish. \qed
\end{proof}

\begin{corollary}
If in addition $m_0(x)$ and $n_0(x)$ are everywhere nonnegative
(nonpositive), then the solution $(u,v)$ (if nontrivial) loses its
compactness immediately.
\end{corollary}
\begin{proof} Indeed, in order for an nontrivial solution to stay
compact one needs $E^u_{\pm}(t)\equiv 0$, $E^v_{\pm}(t)\equiv 0$
for all $t\in [0,T]$. However from Lemma \ref{positivity} it
follows that $m(x,t)$ and $n(x,t)$ remain everywhere nonnegative
(nonpositive) and thus the quantities $E^u_{\pm}(t)$,
$E^v_{\pm}(t)$ defined e.g. in  (\ref{EuDef}) are positive
(negative) for all $t\in (0,T]$ in the case of a nontrivial
solution. \qed
\end{proof}

From (\ref{CH1'}) we know that $L=u+v $ is a density of a
conserved quantity and as such it deserves a special attention.
From Theorem \ref{asymptotics} one can find the asymptotics of $L$
as $x\to \pm \infty$ as $$L \to \frac{1}{2}E_{\pm}(t)e^{-|x|}  $$
where $E_{\pm}\equiv E^u_{\pm}+E^v_{\pm}$. Since the nature of the
solution that we expect is several coupled 'waltzing' waves, i.e.
the maximum elevations of $u(x,t)$ and $v(x,t)$ increase and
decrease with time in the waltzing process. In other words the
functions $E^u_{\pm}(t)$ and $E^v_{\pm}(t)$ are in general
non-monotonic functions of $t$. However in some cases a monotonic
property holds for the conserved density $L$:

\begin{theorem}\label{Epm}
If $(u,v)$ is an initially compactly supported solution and in
addition $m_0(x)$ and $n_0(x)$ are everywhere nonnegative
(nonpositive), then the quantity $E_{+}(t)$ is a monotonically
increasing function and $E_{-}(t)$ is a monotonically decreasing
function.
\end{theorem}
\begin{proof} Indeed, from Lemma \ref{positivity} it
follows that $m(x,t)$ and $n(x,t)$ remain everywhere nonnegative
(nonpositive) and from the explicit form of the inverse Helmholtz
operator $u(x,t)$ and $v(x,t)$  remain everywhere nonnegative
(nonpositive). Since $m(\cdot,t)$ is supported in the interval
$[\xi(\alpha,t),\xi(\beta,t)]$, for each fixed $t$, the derivative
is given by
\begin{equation*}
\frac{\mathrm dE^u_+(t)}{\mathrm dt} =
\int_{\xi(\alpha,t)}^{\xi(\beta,t)} e^y m_t(y,t)\mathrm dy=
\int_{-\infty}^\infty e^y m_t(y,t)\mathrm dy.
\end{equation*}
Similarly, if we define
\[
E^v_+(t)=\int_{\varphi(\alpha,t)}^{\varphi(\beta,t)}e^ym(y,t) \,
dy \quad \mbox{ and }
E^v_-(t)=\int_{\varphi(\alpha,t)}^{\varphi(\beta,t)}e^{-y}m(y,t)dy,
\] then  $E^v_+(0)=E^v_-(0)=0$, and
\begin{equation*}
\frac{\mathrm dE^v_+(t)}{\mathrm dt} = \int_{-\infty}^\infty e^y
n_t(y,t)\mathrm dy.
\end{equation*}
From \ref{CH2} and integration by parts we have
\begin{subequations}
\begin{align}
& \frac{\mathrm dE_+(t)}{\mathrm dt}
=\int_{-\infty}^{\infty} e^y(m_t(y,t)+n_t(y,t))\,dx  \nonumber \\
\nonumber
&=-\int_\mathbb R e^x \left(2v_x(u-u_{xx})+v(u-u_{xx})_x+2u_x(v-v_{xx})+u(v-v_{xx})_x\right)\,dx \nonumber\\
&= \int_{-\infty}^\infty e^y \left(2uv+u_{x}v_x\right)\,dy,\ \quad
t\in[0,T), \nonumber \\
\end{align} \label{IbyParts} \end{subequations}
where all boundary terms after integration by parts vanish, since
the functions {$m(\cdot,t)$}, {$n(\cdot,t)$} have compact support
and {$u(\cdot,t)$}, {$v(\cdot,t)$} decay exponentially at $\pm
\infty$, for all $t\in[0,T)$. Using (\ref{uANDuxdecomp}) for
$u=I_1^u+I_2^u$, $u_x=-I_1^u+I_2^u$, $v=I_1^v+I_2^v$,
$v_x=-I_1^v+I_2^v$, and noticing that all integrals
$I_{1,2}^{u,v}$ are all nonnegative (nonpositive), we have that
$$2uv+u_{x}v_x=3I_1^uI_1^v+I_2^uI_1^v+I_1^uI_2^v+3I_2^uI_2^v$$
and  thus \begin{equation} \frac{\mathrm
dE_+(t)}{\mathrm dt}>0.\label{e1} \end{equation} Similarly, we
have
\begin{align}\nonumber \frac{\mathrm dE_-(t)}{\mathrm dt} =
\int_{-\infty}^{\infty} e^{-y}(m_t(y,t)+n_t(y,t))\,dx \\
\label{e2} = -\int_{-\infty}^\infty e^{-y}
\left(2uv+u_{x}v_x\right)\,dy<0,\ \ t\in[0,T) \end{align} for
analogous reasons as before. \qed
\end{proof}

\subsection{Evolution in the case $u=\bar{ v}$ when initially compactly supported}

Some analytical results can be established in the case $u=\bar{
v}$, for example one can prove immediately the analogue of Theorem
\ref{Epm}:

\begin{theorem}
If $u=\bar{ v}$ is initially compactly supported, then $E_-=(E^u_-
+ E^v_-)(t)$ is a decreasing function, with $E_-(0)=0$, and
$E_+(t)$ is increasing, with $E_+(0)=0$.
\end{theorem}

\begin{proof} Follows the lines of the proof of Theorem \ref{Epm}.
In this case $2uv+u_xv_x=2|u|^2+|u_x|^2 \ge 0$ and for nontrivial
solutions this expresion is at least somewhere positive. \qed
\end{proof}

The following Lemma is proved by making extensive  use of relation
\eqref{udecomp}.
\begin{lemma}\cite{HenJNMP}\label{equivalence}
Let $(u,v)$ be a solution of system \eqref{sys}, and suppose $u$ is such that
$m=u-u_{xx}$ has compact support. Then, for each fixed time
$0<t<T$, $u$ has compact support if and only if
\begin{equation}\int_\mathbb Re^xm(x)\,dx=\int_\mathbb
Re^{-x}m(x)\,dx=0.  \label{zeroint} \end{equation}  The equivalent
relation holds for the functions $v$ and $n$. \end{lemma}

We now establish a relation which is satisfied by solutions of
\eqref{sys} whose support remains compact throughout their
evolution. This relation will have profound implications for
solutions $(u,v)$ of \eqref{sys} which have a direct relation to
each other, as we see in Corollary \eqref{c1}.
\begin{theorem}
\label{unotcompact} Let us assume that the functions $u_0,v_0$
have compact support, and let $T>0$ be the maximal existence time
of the solutions $u(x,t),v(x,t)$ which are generated by this
initial data. If, for every $t\in[0,T)$, the function $x\mapsto
\left(u(x,t),v(x,t)\right)$ has compact support, then
\begin{equation} \label{relunot}\int_\mathbb R e^{x} \left(2uv+u_{x}v_x\right)\,dx=\int_\mathbb R e^{-x} \left(2uv+u_{x}v_x\right)\,dx=0 \ \ \mbox{ for } t\in[0,T). \end{equation}
\end{theorem}
\begin{proof}
By the assumptions of this theorem,  Lemma~\ref{equivalence}
applies. Using \eqref{sys} and differentiating the left hand side
of \eqref{zeroint} with respect to $t$ we get
\begin{align*}
\frac{d}{dt}\int_\mathbb R e^x \left(m+n\right)\,dx  = -\int_\mathbb R e^x
\left(2v_xm+vm_x+2u_xn+un_x\right)\,dx \\
=\int_\mathbb R e^x \left(2uv+u_{x}v_x\right)\,dx=0,
 \end{align*}
similarly to the proof of Theorem \ref{Epm}. The final equality
follows from the fact that identity \eqref{zeroint} holds for all
$t\in[0,T)$, by Lemma \ref{equivalence}.

Similarly, we get
\begin{equation}
\frac{d}{dt}\int_\mathbb R e^{-x}
\left(m+n\right)\,dx=-\int_\mathbb R e^{-x}
\left(2uv+u_{x}v_x\right)\,dx=0.
\end{equation}
 Therefore,
\begin{equation}\label{r1}\int_\mathbb R e^{x} \left(2uv+u_{x}v_x\right)\,dx=\int_\mathbb R e^{-x} \left(2uv+u_{x}v_x\right)\,dx=0 \quad t\in[0,T). \end{equation}
The expression under the integral on the right hand side of this
relation must be identically zero by \eqref{zeroint}. This
completes the proof. \qed
\end{proof}
\begin{corollary}\label{c1}
Let us suppose that $u(x,t)=\bar{ v}(x,t)$.  Then the only
solution $(u,v)$ of \eqref{sys} which is compactly supported over
a positive time interval is the trivial solution $u\equiv v\equiv
0$. That is to say, any non-trivial solution $(u,v)$ of
\eqref{sys} which is initially compactly supported instantaneously
loses this property, and so has an infinite propagation speed.
\end{corollary}
\begin{proof}
The statement follows directly from relations in \eqref{r1}. \qed
\end{proof}
%\subsubsection{Other relations...}
%If we have
%\[
%u=w+e, \qquad v=w-e,
%\]
%then
%\[
%\int_\mathbb R e^x \left(2uv+u_{x}v_x\right)\,dx
%\] becomes
%\begin{align*}
%\int_\mathbb R e^x \left(2(w+e)(w-e)+(w+e)_x(w-e)_x\right)\,dx
%\\ =\int_\mathbb R e^x \left(2w^2-2e^2+w_x^2-e^2_x\right)\,dx
%\end{align*}
%and so the relations in \eqref{r1} imply that, if $(u,v)$ is
%compactly supported for all times $t\in[0,T)$, then
%\begin{align*}
%\int_\mathbb R e^x (w^2+\frac 12 w_x^2)\,dx =\int_\mathbb R e^x (e^2+\frac 12 e_x^2)\,dx, \\
%\int_\mathbb R e^{-x} (w^2+\frac 12 w_x^2)\,dx =\int_\mathbb R
%e^{-x} (e^2+\frac 12 e_x^2)\,dx,
%\end{align*}
%therefore the terms $w$ and $e$ must have equal weighted energy
%norms.

\subsection{Global solutions for nonnegative $m_0,n_0$}

From (\ref{udecomp}) and (\ref{uxdecomp}) it follows that
\begin{equation} u(x,t)+u_x(x,t)= e^x\int_x^{\infty}e^{-y}m(y,t)\,dy. \label{uplusuxdecomp}
\end{equation}

Thus the nonnegativity of $m(x,t)$, $n(x,t)$,  ensures $ u_x
(x,t)\ge -u(x,t)$ and similarly  $v_x(x,t)\ge -v(x,t)$, preventing
blowup in finite time, because the solution $(u,v)$ is uniformly
bounded as long as it exists.

Blowup however might be possible if $m(x,0)$, $n(x,0)$ take both
positive and negative values.

%**** It might be easier to work with $L(x,t)\equiv n(x,t)+m(x,t)$.
% If $L$ breaks then either $u$ breaks or $v$ or both.

\section{Conclusions}

In the presented study we analysed the behavior of the solutions
of the CCCH system when $m,n$ are initially compactly supported
and (i) initially $u,v$ everywhere nonpositive/nonnegative (ii)
$u=\bar{v}$. In both cases the result is that the compactness
property is lost immediately, i.e. for any time $t>0$.
Asymptotically the solutions decay exponentially to zero, such
that $u+v$ decays to zero monotonically. The exponential decay is
already observed in the case of the peakon solutions, where $m,n$
are supported only at finite number of points.

\section{Acknowledgments}
We are grateful to our friend and colleague James Percival for
providing us the figure.  The work of RII is supported by the
Science Foundation Ireland (SFI), under Grant No. 09/RFP/MTH2144.
The work by DDH was partially supported by Advanced Grant 267382 FCCA from the European Research Council.

\label{lastpage}

\end{document}